\documentclass[reqno,11pt]{amsart}

\usepackage{tikz}
\usetikzlibrary{matrix,arrows,calc,patterns,snakes,decorations.markings}

\usepackage[mathscr]{eucal}
\usepackage{amscd}
\usepackage{amsfonts}
\usepackage{amsmath, amsthm, amssymb, bbm, bm}
\usepackage{stmaryrd}
\usepackage{lscape}

\usepackage{latexsym}
\usepackage{graphics}

\theoremstyle{plain}
\newtheorem{thm}{Theorem}[section]
\newtheorem{cor}[thm]{Corollary}
\newtheorem{lem}[thm]{Lemma}

\theoremstyle{definition}

\newtheorem{defn}[thm]{Definition}
\newtheorem{ex}[thm]{Example}

\theoremstyle{remark}
\newtheorem{rem}[thm]{Remark}
\theoremstyle{setup}

\numberwithin{equation}{section}

\newcommand{\ul}[1]{\underline{#1}}

\DeclareMathAlphabet{\mathpzc}{OT1}{pzc}{m}{it}

\DeclareMathOperator{\Perf}{\mathsf{Perf}}

\def\GP{\mathop{\rm GP}\nolimits}

\DeclareMathOperator{\rad}{\mathsf{rad}}

\DeclareMathOperator{\Coh}{\mathsf{Coh}}

\DeclareMathOperator{\MCM}{\mathsf{MCM}}

\DeclareMathOperator{\injdim}{\mathsf{inj.dim}}

\newcommand{\ind}{\opname{ind}}

\DeclareMathOperator{\Hom}{\mathsf{Hom}}

\DeclareMathOperator{\Ext}{\mathsf{Ext}}

\DeclareMathOperator{\End}{\mathsf{End}}

\newcommand{\opname}[1]{\operatorname{\mathsf{#1}}}

\renewcommand{\mod}{\opname{mod}\nolimits}

\newcommand{\proj}{\opname{proj}\nolimits}

\newcommand{\ca}{{\mathcal A}}

\newcommand{\cc}{{\mathcal C}}
\newcommand{\cd}{{\mathcal D}}

\newcommand{\cf}{{\mathcal F}}

\newcommand{\co}{{\mathcal O}}
\newcommand{\cp}{{\mathcal P}}

\newcommand{\ct}{{\mathcal T}}

\newcommand{\ra}{\rightarrow}

\input xy
\xyoption{all}

\newcommand{\bsm}{\begin{smallmatrix}}
\newcommand{\esm}{\end{smallmatrix}}

\renewcommand{\mod}{\mathsf{mod}}

\renewcommand{\AA}{\mathbb{A}}

\newcommand{\XX}{\mathbb{X}}

\title{Singularity categories of gentle algebras}
\author{Martin Kalck}
\thanks{The author gratefully acknowledges support by the DFG grant Bu--1866/2--1.}
\address{Martin Kalck, The Maxwell Institute, School of Mathematics, James Clerk Maxwell Building, The King's Buildings, Mayfield Road, Edinburgh, EH9 3JZ, UK.}
\email{m.kalck@ed.ac.uk}
\begin{document}

\begin{abstract}
We determine the singularity category of an arbitrary finite dimensional gentle algebra $\Lambda$. It is a finite product of $n$-cluster categories of type $\mathbb{A}_{1}$. Equivalently, it may be described as the stable module category of a selfinjective gentle algebra. If $\Lambda$ is a Jacobian algebra arising from a triangulation $\ct$ of an unpunctured marked Riemann surface, then the number of factors equals the number of inner triangles of $\ct$. 
\end{abstract}

\maketitle

\section{Introduction}
Singularity categories were introduced and studied by Buchweitz \cite{Buchweitz}. Recently, Orlov's global version \cite{Orlov2} attracted a lot of interest in algebraic geometry and theoretical physics: in particular, its relation to Kontsevich's Homological Mirror Symmetry Conjecture. 

For Iwanaga--Gorenstein rings, Buchweitz gave an equivalent description of singularity categories in terms of stable categories of Gorenstein projective modules (also known as maximal Cohen--Macaulay modules), see \cite{Buchweitz} and also Happel \cite{Happel}, Keller \& Vossieck \cite{KV} and Rickard \cite{Rickard}. In particular, singularity categories of selfinjective algebras are equivalent to their stable module categories, which were thoroughly studied in representation theory. X.-W. Chen \cite{Chen} described the singularity categories of artin algebras with radical square zero in terms of projective modules over certain von Neumann regular algebras. He shows that their underlying additive categories are semisimple abelian categories.

 The aim of this note is to describe the singularity categories of another class of finite dimensional algebras - so called gentle algebras (Definition \ref{D:Gentle}). As it turns out, their underlying additive categories are again semisimple. Examples of gentle algebras include tilted algebras of type $\AA_{n}$ \cite{Assem82} and $\widetilde{\AA}_{n}$ \cite{AssemSkowronski} and more generally all algebras which are derived equivalent to gentle algebras \cite{SchroerZimmermann}. Moreover, algebras derived equivalent to $\mathbb{A}_{n}$-configurations of projective lines \cite{Burban} and Jacobian algebras coming from triangulations of unpunctured marked surfaces are gentle. Furthermore, cluster tilted algebras of type $\mathbb{A}_{n}$ and $\widetilde{\AA}_{n}$ are gentle - in fact, they arise from unpunctured marked discs and annuli \cite{ABCP}.
 
Our proof combines Buchweitz' equivalence with the explicit classification of indecomposable modules over gentle algebras which follows (see e.g.~\cite{WW, ButlerRingel}) from work of Ringel \cite{Ringel}, who builds on techniques developed by Gelfand \& Ponomarev \cite{GP} in their study of indecomposable representations of the Lorentz group. More precisely, indecomposable modules are either \emph{string} or \emph{band} modules. Band modules are never submodules of projective modules - in particular, they cannot be Gorenstein projective (GP). We show that string modules are GP precisely if they are projective or left ideals generated by certain arrows. We complete our description of the singularity categories by proving that all non-trivial morphisms between indecomposable GPs factor over projectives. % Moreover, the class of gentle algebras is closed under derived equivalence, by Schr\"oer \& Zimmermann \cite{SchroerZimmermann} and their (derived) module categories are well understood, e.g.~there is a complete classification of indecomposable objects in both categories - this builds on techniques developed by Gelfand \& Ponomarev \cite{GP} in their study of indecomposable representations of the Lorentz group, see e.g.~Wald \& Waschb\"usch \cite{WW} and  Butler \& Ringel \cite{ButlerRingel} for module categories and Burban \& Drozd \cite{BurbanDrozd} for bounded derived categories.
\section{Definitions and main result}\label{S:DefMain}
\noindent Let $k$ be an algebraically closed field and let $Q$ be a finite quiver with set of arrows $Q_{1}$. We read elements in the path algebra $kQ$ from right to left. 
\begin{defn}\label{D:Gentle}
A \emph{gentle algebra} is a finite dimensional $k$-algebra $\Lambda=kQ/I$ such that:
\begin{itemize}
\item[(G1)]  At any vertex, there are at most two incoming and at most two outgoing arrows.
\item[(G2)] $I$ is a two-sided admissable ideal, which is generated by paths of length two.
\item[(G3)] For each arrow $\beta \in Q_{1}$,  there is at most one arrow $\alpha \in Q_{1}$ such that 
$0 \neq \alpha\beta \in I$ and at most one arrow $\gamma \in Q_{1}$ such that 
$0 \neq \beta \gamma \in I.$
\item[(G4)] For each arrow $\beta\in Q_{1}$, there is at most one arrow $\alpha \in Q_{1}$ such that 
$\alpha\beta \notin I$ and at most one arrow $\gamma \in Q_{1}$ such that 
$\beta \gamma \notin I.$
\end{itemize}
\end{defn}

\begin{rem}
It is well-known that gentle algebras can be described more conceptually as those finite dimensional algebras with special biserial repetitive algebra, see \cite{AssemSkowronski} and \cite{PS}.
\end{rem}

\begin{ex} \label{Ex:Illustrative} An example of a gentle algebra $\Lambda=kQ/I$  is given by the quiver $Q$
\[
\begin{tikzpicture}[description/.style={fill=white,inner sep=2pt}]
    \matrix (n) [matrix of math nodes, row sep=1.1em,
                 column sep=2.25em, text height=1.5ex, text depth=0.25ex,
                 inner sep=0pt, nodes={inner xsep=0.3333em, inner
ysep=0.3333em}]
    {  1&&2 && 3 && 4 \\
       5&& 6 &&7&& 8 \\
    };
\draw[->] (n-1-1) edge node[fill=white, scale=0.75, yshift=0mm] [midway] {$a$} (n-1-3); 
\draw[->] (n-1-3) edge node[fill=white, scale=0.75, yshift=0mm] [midway] {$b$} (n-1-5);
\draw[->] (n-1-5) edge node[fill=white, scale=0.75, yshift=0mm] [midway] {$c$} (n-1-7);

\draw[->] (n-2-1) edge node[fill=white, scale=0.75, xshift=-3mm] [midway] {$d$} (n-1-1);
\draw[->] (n-2-3) edge node[fill=white, scale=0.75, xshift=-3mm] [midway] {$e$} (n-1-3);

\draw[->] (n-1-3) edge node[fill=white, scale=0.75, yshift=0mm] [midway] {$f$} (n-2-5);
\draw[->] (n-1-7) edge node[fill=white, scale=0.75, yshift=0mm] [midway] {$g$} (n-2-5);

\draw[->] (n-2-7) edge node[fill=white, scale=0.75, xshift=3mm] [midway] {$h$} (n-1-7);

\draw[<-] (n-2-3) edge node[fill=white, scale=0.75, yshift=0mm] [midway] {$j$} (n-2-5);
\draw[<-] (n-2-7) edge node[fill=white, scale=0.75, yshift=0mm] [midway] {$k$} (n-2-5);

\draw[<-] (n-2-1) edge node[fill=white, scale=0.75, yshift=0mm] [midway] {$i$} (n-2-3);   

\end{tikzpicture}
\]
with two-sided ideal $I$ generated by the paths $ba$, $fe$, $jf$, $ej$, $kg$, $hk$ and $gh$.
\end{ex}

Gei{\ss} \& Reiten \cite{GeissReiten} have shown that gentle algebras are \emph{Iwanaga--Gorenstein rings}, i.e.~they have finite injective dimension as left and as right modules over themselves. For any Iwanaga--Gorenstein ring $R$, Zaks \cite{Zaks} has shown that  $\injdim _{R}R=d=\injdim R_{R}$ holds. Following Buchweitz, we call $d$ the \emph{virtual dimension} of $R$ -  for commutative local Noetherian rings, it coincides with the \emph{Krull dimension}.  Inside the category $R-\mod$ of all finite dimensional left $R$-modules, the full subcategory of \emph{Gorenstein projective} $R$-modules
 \begin{align}
\GP(R)=\bigl\{M \in R-\mod\, \big| \Ext^i_{R}(M, R)=0 \text{  for all } i>0 \bigr\}
\end{align}
 is of special interest. Let $M$ and $N$ be finite dimensional left $R$-modules.
We list some well-known facts about Gorenstein projective $R$-modules, see e.g.~Buchweitz \cite{Buchweitz}. \begin{itemize}
\item[(GP1)] A GP $R$-module is either projective or of infinite projective dimension.
\item[(GP2)] $M$ is GP if and only if $M \cong \Omega^d(N)$ for some $N$, where $d$ is the virtual dimension. In particular, every GP module is a submodule of a projective module.
\item[(GP3)] $\GP(R)$ is a Frobenius category with $\proj \GP(R)=\proj-R$.
\end{itemize}
Moreover, the embedding $\GP(R) \subseteq \cd^b(\mod-R)$ induces a triangle equivalence (see  \cite{Buchweitz})
\begin{align}\label{E:Buchweitz}
\frac{\GP(R)}{\proj-R}=:\underline{\GP}(R) \longrightarrow \cd_{sg}(R):=\frac{\cd^b(\mod-R)}{K^b(\proj-R)},
\end{align}
where the triangulated quotient category $\cd_{sg}(R)$ is called the \emph{singularity category} of $R$, see \cite{Buchweitz} and also \cite{Orlov2}. The additive quotient category $\underline{\GP}(R)$ is called the \emph{stable category} of Gorenstein projective $R$-modules. It admits a triangulated structure by Happel's general result on stable categories of Frobenius categories \cite{Happel88}. More precisely, $\underline{\GP}(R)$ has the same objects as $\GP(R)$. Two morphisms in $\GP(R)$ are identified in $\underline{\GP}(R)$ if their difference factors over a projective $R$-module. Moreover, in $\underline{\GP}(R)$ the inverse shift functor $[-1]$ is given by the syzygy functor $\Omega$. 

In order to state the main result of this note, we need to introduce some notations: for a gentle algebra $\Lambda=kQ/I$, we denote by $\cc(\Lambda)$ the set of equivalence classes (with respect to cyclic permutation) of repetition-free cyclic paths $\alpha_{1}\ldots \alpha_{n}$ in $Q$ such that $\alpha_{i}\alpha_{i+1} \in I$ for all $i$, where we set $n+1=1$. Property (G3) implies that for every arrow $\alpha \in Q_{1}$, there is at most one cycle $c \in \cc(\Lambda)$ containing it. Moreover, we write $l(c)$ for the \emph{length} of a cycle $c \in \cc(\Lambda)$, i.e.~$l(\alpha_{1}\ldots \alpha_{n})=n$. We define $R(\alpha)$ to be the \emph{left} \,Êideal $\Lambda\alpha$ generated by $\alpha$. It follows from the definition of gentle algebras that this is a direct summand of the radical $\rad P_{s(\alpha)}$ of the indecomposable projective  $\Lambda$-module $P_{s(\alpha)}=\Lambda e_{s(\alpha)}$, where $s(\alpha)$ is the \emph{start point} of $\alpha$. In fact, all radical summands of indecomposable projectives arise in this way. Moreover, the radicals of indecomposable projectives decompose into at most two direct summands by (G1), see e.g.~\eqref{E:IndProjGentle} for an illustration.
\begin{ex}\label{Ex:Continued}
In Example \ref{Ex:Illustrative}, we have $\cc(\Lambda)=\{jfe, kgh\}$ and
\[
\begin{tikzpicture}[description/.style={fill=white,inner sep=2pt}]
    \matrix (n) [matrix of math nodes, row sep=0.4em,
                 column sep=2.25em, text height=1.5ex, text depth=0.25ex,
                 inner sep=0pt, nodes={inner xsep=0.3333em, inner
ysep=0.1111em}
]
    {   &6& 5&1 &2&7&8 \\
        7&& & &   \\
       &8 \\
    };
\draw[->] (n-2-1) edge node[fill=white, scale=0.75, yshift=0mm] [midway] {$j$} (n-1-2); 
\draw[->] (n-2-1) edge node[fill=white, scale=0.75, yshift=0mm] [midway] {$k$} (n-3-2); 
   
\draw[->] (n-1-2) edge node[fill=white, scale=0.75, yshift=0mm] [midway] {$i$} (n-1-3);  
\draw[->] (n-1-3) edge node[fill=white, scale=0.75, yshift=0mm] [midway] {$d$} (n-1-4); 
\draw[->] (n-1-4) edge node[fill=white, scale=0.75, yshift=0mm] [midway] {$a$} (n-1-5); 
\draw[->] (n-1-5) edge node[fill=white, scale=0.75, yshift=0mm] [midway] {$f$} (n-1-6); 
\draw[->] (n-1-6) edge node[fill=white, scale=0.75, yshift=0mm] [midway] {$k$} (n-1-7); 

\draw[decoration={brace, amplitude=2.5mm}, decorate] ($(n-1-2.west) +(0, 3mm)$) -- node[ scale=0.75, yshift=6.5mm] [midway] {$R(j)$} ($(n-1-7.east)+(0, 3mm)$);

\draw[decoration={brace, amplitude=1.2mm}, decorate] ($(n-3-2.east) +(0, -2mm)$) -- node[ scale=0.75, yshift=-5mm] [midway] {$R(k)$} ($(n-3-2.west)+(0, -2mm)$);

\end{tikzpicture}
\]
describes the indecomposable projective $\Lambda$-module $P_{7}=\Lambda e_{7}$ and its radical summands $R(k)=\Lambda k$ and $R(j)=\Lambda j$. We note, that there is a non-zero morphism $R(k) \ra R(j)$. However, it factors over the projective $P_{7}$ and thus vanishes in the stable category. 

\end{ex}

The following theorem is the main result of this note:

\begin{thm}\label{P:Main}
Let $\Lambda=kQ/I$ be a finite dimensional gentle algebra. Then 
\begin{itemize}
\item[(a)] $\ind \GP(\Lambda) = \ind \proj-\Lambda \cup \{ R(\alpha_{1}), \ldots, R(\alpha_{n}) \big| c= \alpha_{1}\ldots\alpha_{n} \in \cc(\Lambda) \},$

\noindent
where $\ind$ denotes the set of isomorphism classes of indecomposable objects.
\item[(b)] There is an equivalence of triangulated categories 
\begin{align}
\cd_{sg}(\Lambda) \cong \prod_{c \in \cc(\Lambda)} \frac{\displaystyle \cd^b(k-\mod)}{\displaystyle [l(c)]},
\end{align}
where $\cd^b(k-\mod)/[l(c)]$ denotes the triangulated orbit category, see Keller \cite{Keller}. \end{itemize}
\end{thm}

We prove this result in Section \ref{S:Proof}.

\begin{rem}\label{R:Main}
The triangulated orbit category $\cd^b(k-\mod)/[n]$ is also known as the $(n-1)$-\emph{cluster category} of Dynkin-type $\mathbb{A}_{1}$, see e.g.~H.~Thomas \cite{Thomas}. Moreover, it is triangle equivalent to the \emph{stable module category} $I_{n}-\ul{\mod}$ of the selfinjective gentle algebra $I_{n}=kC_{n}/\ca^2$, where the quiver $C_{n}$ is an oriented cycle with $n$ vertices and $\ca \subseteq kC_{n}$ is the two-sided ideal generated by all arrows in $C_{n}$. The $I_{n}$ are \emph{uniserial} (or \emph{Nakayama}) algebras and are in fact the only indecomposable gentle algebras which are \emph{selfinjective}.

As additive categories, the orbit categories $\cd^b(k-\mod)/[n]$ are equivalent to the semisimple abelian categories $k^n-\mod$. In particular, the singularity categories of gentle algebras are semisimple abelian when viewed as additive categories. Another class of finite dimensional algebras with semisimple singularity categories are the algebras with radical square zero, see X.-W. Chen \cite{Chen}.
\end{rem}

\section{Applications and Examples} \label{S:Examples}

\begin{cor}\label{C:Invariant} Let $\Lambda$ and $\Lambda'$ be gentle algebras. If there is an equivalence of triangulated categories $\cd^b(\Lambda-\mod) \cong \cd^b(\Lambda'-\mod)$, then there is a bijection of sets
\begin{align}
f \colon \cc(\Lambda) \stackrel{\sim}\longrightarrow \cc(\Lambda'),
\end{align}
such that $l(c)=l(f(c))$ for all $c \in \cc(\Lambda)$.
\end{cor}
\begin{proof}
The derived equivalence $\cd^b(\Lambda-\mod) \cong \cd^b(\Lambda'-\mod)$ yields a triangle equivalence $\cd_{sg}(\Lambda)\cong \cd_{sg}(\Lambda')$. Now Theorem \ref{P:Main} completes the proof.
\end{proof}

\begin{rem}
Corollary \ref{C:Invariant} recovers parts of a derived invariant for gentle algebras, which was introduced by Avella-Alaminos \& Gei{\ss} \cite{GeissAvellaAlaminos}. More precisely, our result shows that 
\begin{align}
\bigl.\phi_\Lambda\bigr|_{0 \times \mathbb{N}} = \bigl.\phi_{\Lambda'}\bigr|_{0 \times \mathbb{N}}\,\, , 
\end{align}
where $\phi_\Lambda, \phi_{\Lambda'}\colon \mathbb{N}^2 \to \mathbb{N}$ are the invariants of \cite{GeissAvellaAlaminos} associated with $\Lambda$ and $\Lambda'$, respectively.
\end{rem}

\begin{rem}
Buan \& Vatne \cite{BuanVatne} show the converse of Corollary \ref{C:Invariant} for two cluster tilted algebras $\Lambda$ and $\Lambda'$ of type $\mathbb{A}_{n}$ for some fixed $n \in \mathbb{N}$. In other words, two such algebras are derived equivalent if and only if their singularity categories are triangle equivalent. This result generalises to $m$-cluster tilted algebras of type $\mathbb{A}_n$ by work of Murphy \cite{Murphy}.
\end{rem}

The following geometric example was pointed out by Igor Burban.
\begin{ex}
Let $\mathbb{X}_{n}$ be a chain of $n$ projective lines  
\begin{align}\label{E:ChainOfProjLines}
\begin{array}{c}
\begin{tikzpicture} 
\draw (0,0,0) to [bend left=25]  (2,0,0);
\node[scale=0.75] at (1, 0.5, 0) {$C_{1}$};
\node[scale=0.75] at (1.75, -0.25, 0) {$s_{1}$};
\draw (1.5,0,0) to [bend left=25]  (3.5,0,0);
\node[scale=0.75] at (2.5, 0.5, 0) {$C_{2}$};
\node at (4, 0, 0) {$\cdots$}; 
\draw (4.5,0,0) to [bend left=25]  (6.5,0,0);
\node[scale=0.75] at (5.5, 0.5, 0) {$C_{n-2}$};
\node[scale=0.75] at (6.25, -0.25, 0) {$s_{n-2}$};
\draw (6,0,0) to [bend left=25]  (8,0,0);
\node[scale=0.75] at (7, 0.5, 0) {$C_{n-1}$};
\node[scale=0.75] at (7.75, -0.25, 0) {$s_{n-1}$};
\draw (7.5,0,0) to [bend left=25]  (9.5,0,0);
\node[scale=0.75] at (8.5, 0.5, 0) {$C_{n}$};
\end{tikzpicture}
\end{array}
\end{align}
Using Buchweitz' equivalence \eqref{E:Buchweitz}  and Orlov's localization theorem \cite{Orlov1}, the singularity category of $\mathbb{X}_{n}$ may be described as follows
\begin{align}\label{E:SingChain}
\bigl( \cd_{sg}(\mathbb{X}_{n}) \bigr)^{\omega}:= \left(\frac{\displaystyle \cd^b(\Coh \mathbb{X}_{n})}{\displaystyle \Perf(\XX_{n})}\right)^\omega \cong \bigoplus_{i=1}^{n-1} \ul{\MCM}\bigl(O_{nd}\bigr) \cong  \bigoplus_{i=1}^{n-1} \frac{\displaystyle \cd^b(k-\mod)}{\displaystyle [2]},
\end{align}
where $(-)^{\omega}$ denotes the idempotent completion \cite{BalmerSchlichting} and $\ul{\MCM}(O_{nd})$ denotes the stable category of maximal Cohen--Macaulay modules over the nodal singularity $O_{nd}=k\llbracket x, y \rrbracket/(xy)$. 

In particular, there is a fully faithful triangle functor
\begin{align}\label{E:WithoutCompl}
\cd_{sg}\bigl(\mathbb{X}_{n}\bigr) \longrightarrow \bigoplus_{i=1}^{n-1} \ul{\MCM}\bigl(O_{nd}\bigr),
\end{align}
 which is induced by
\begin{align}\label{E:LocalizationChain}
\cd^b(\Coh \mathbb{X}_{n}) \ni \cf  \longmapsto \left(\widehat{\cf}_{s_{1}}, \ldots,\widehat{\cf}_{s_{n-1}}\right) \in \bigoplus _{i=1}^{n-1} O_{nd}-\mod,
\end{align}
where $s_{1}, \ldots, s_{n-1}$ denote the singular points of $\mathbb{X}_{n}$.
For $1\leq l \leq m \leq n$, let $\co_{[l,m]}$ be the structure sheaf of the subvariety $\bigcup_{k=l}^m C_{k} \subseteq \mathbb{X}_{n}$. Here, the $C_{i}$ denote the irreducible components of $\mathbb{X}_{n}$ as shown in \eqref{E:ChainOfProjLines}. Then \eqref{E:LocalizationChain} maps $\co_{[1, i]}$ to $(O_{nd}, \ldots, O_{nd}, k\llbracket x \rrbracket, 0, \ldots, 0)$ and $\co_{[j, n]}$ to $( 0, \ldots, 0, k\llbracket y \rrbracket, O_{nd}, \ldots, O_{nd})$, where $k\llbracket x \rrbracket$  and  $k\llbracket y \rrbracket$ are located in the $i$-th and $j$-th place, respectively. In particular, the functor in \eqref{E:WithoutCompl} is essentially surjective. Therefore, the singularity category $\cd_{sg}(\mathbb{X}_{n})$ is idempotent complete.

We explain an alternative approach to obtain the equivalence \eqref{E:WithoutCompl}, which uses and confirms Theorem \ref{P:Main}. Burban \cite{Burban} showed  that $\cd^b(\Coh \XX_{n})$ has a tilting bundle with endomorphism algebra $\Lambda_{n}$ given by the following quiver
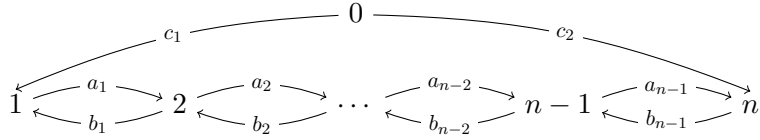
\[
\begin{tikzpicture}[description/.style={fill=white,inner sep=2pt}]
    \matrix (n) [matrix of math nodes, row sep=1.7em,
                 column sep=2.25em, text height=1.5ex, text depth=0.25ex,
                 inner sep=0pt, nodes={inner xsep=0.3333em, inner
ysep=0.3333em}]
    {  &&&& 0 \\
       1 && 2 && \cdots && n-1 && n \\
    };
\draw[->] (n-1-5) edge  [bend right=10] node[fill=white, scale=0.75, yshift=0mm] [midway] {$c_{1}$} ($(n-2-1.north)+(0, -0.75mm)$); 
  
\draw[->] (n-1-5)  edge [bend left=10]  node[fill=white, scale=0.75, yshift=0mm] [midway] {$c_{2}$} ($(n-2-9.north)+(0, -1mm)$); 

\draw[->] (n-2-1) edge [bend left=20] node[fill=white, scale=0.75, yshift=0mm] [midway] {$a_{1}$} (n-2-3); 
\draw[->] (n-2-3) edge [bend left=20] node[fill=white, scale=0.75, yshift=0mm] [midway] {$a_{2}$} ($(n-2-5.west)+(0, 1mm)$); 
\path[->]  ($(n-2-5.east)+(0, 1mm)$) edge [bend left=20] node[fill=white, scale=0.75, yshift=0mm] [midway] {$a_{n-2}$} ($(n-2-7.west)+(0, 1mm)$); 
\path[->] ($(n-2-7.east)+(0, 1mm)$) edge [bend left=20] node[fill=white, scale=0.75, yshift=-0.75mm] [midway] {$a_{n-1}$}  ($(n-2-9.west)+(0, 1mm)$); 

\draw[<-] (n-2-1) edge [bend right=20] node[fill=white, scale=0.75, yshift=0mm] [midway] {$b_{1}$} (n-2-3); 
\draw[<-] (n-2-3) edge [bend right=20] node[fill=white, scale=0.75, yshift=0mm] [midway] {$b_{2}$} ($(n-2-5.west)+(0, -1mm)$); 
\path[<-]  ($(n-2-5.east)+(0, -1mm)$) edge [bend right=20] node[fill=white, scale=0.75, yshift=0mm] [midway] {$b_{n-2}$} ($(n-2-7.west)+(0, -1mm)$); 
\path[<-] ($(n-2-7.east)+(0, -1mm)$) edge [bend right=20] node[fill=white, scale=0.75, yshift=1mm] [midway] {$b_{n-1}$} ($(n-2-9.west)+(0, -1mm)$); 

%\draw[dotted, -] (n-1-1.east) -- (n-1-2.west);    
%\draw[->] (n-1-2.east) -- node[scale=0.75, yshift=3mm] [midway] {$H^i(\varphi)$} (n-1-3.west);  
%\draw[->] (n-1-3.east) -- node[scale=0.75, yshift=3mm] [midway] {$H^i(f')$} (n-1-4.west); 
%\draw[->] (n-1-4.east) -- (n-1-5.west); 
%\draw[dotted, -] (n-1-5.east) -- (n-1-6.west); 
\end{tikzpicture}
\]
with relations $a_{i}b_{i}=0=b_{i}a_{i}$ for all $1 \leq i \leq n-1$. Hence we have a triangle equivalence $\cd^b(\Coh \XX_{n}) \ra \cd^b(\mod-\Lambda_{n})$ inducing an equivalence of triangulated categories
\begin{align}
\cd_{sg}(\XX_{n}) \xrightarrow{\sim} \cd_{sg}(\Lambda_{n}).
\end{align}
Since $\Lambda_{n}$ is a gentle algebra, we can apply Theorem \ref{P:Main}. $\cc(\Lambda_{n})$ consists of $n-1$ cycles of length two. Therefore $\cd_{sg}(\Lambda_{n})$ is equivalent to the right hand side of \eqref{E:SingChain}. In particular, we see again that the singularity category $\cd_{sg}(\XX_{n})$ is idempotent complete.
\end{ex}

Assem, Br\"ustle, Charbonneau-Jodoin \& Plamondon \cite{ABCP} studied a class of gentle algebras $A(S, \ct)$ arising from triangulations $\ct$ of marked Riemann surfaces $S$ without punctures. In particular, they show that the `inner triangles' of $\ct$ are in bijection with the elements of $\cc(A(S, \ct))$, which in this case are all of length three. This has the following consequence.

\begin{cor}\label{C:Jacobian}
In the notation above, the number of direct factors of the singularity category $\cd_{sg}(A(S, \ct))$ equals the number of inner triangles of $\ct$.  
\end{cor}

\begin{ex}
A prototypical case is the hexagon $S$ with six marked points on the boundary. We consider the following triangulation $\ct$ with exactly one inner triangle.
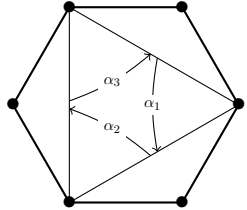
\[
\begin{tikzpicture}
\foreach \l in {1.5}
\foreach \ha in {0.52}
\foreach \hb in {0.48}
\foreach \h in {0.5}
\foreach \cp in {1.5} %=1+cos(60)
\foreach \s in {-0.86602540378} %=-sin(60)
{
\foreach \a in {0, -120} \draw[thick=40mm] (0, 0) -- +(\a:\l); 
\draw[thick=40mm] (\l, 0) -- +(300:\l); 
\draw[thick=40mm] ($\l*(\cp, \s )$) -- +(-120:\l); 
\draw[thick=40mm] ($2*\l*(0.5, \s)$) -- +(180:\l); 
\draw[thick=40mm] ($2*\l*(0, \s)$) -- +(120:\l); 

\draw (0,0) -- ($\l*(\cp, \s)$);
\draw ($2*\l*(0,\s)$) -- ($\l*(\cp, \s)$);
\draw ($2*\l*(0,\s)$) -- (0,0);
\path[->] ($\ha*\l*(\cp, \s)$) edge [bend right=10] node[fill=white, scale=0.55] [midway] {$\alpha_{1}$} ($\hb*\l*(-\cp,\s)+\l*(\cp,\s)$); 

\path[->] ($\ha*\l*(-\cp,\s)+\l*(\cp,\s)$) edge [bend right=10] node[fill=white, scale=0.55] [midway] {$\alpha_{2}$} ($\ha*2*\l*(0, \s)$);
\path[<-] ($\hb*\l*(\cp, \s)$)  edge [bend left=10] node[fill=white, scale=0.55] [midway] {$\alpha_{3}$}($\hb*2*\l*(0, \s)$);

\node at (\l, 0) {$\bullet$};
\node at (0, 0) {$\bullet$};
\node  at ($2*\l*(0.5, \s)$) {$\bullet$};
\node  at ($2*\l*(0, \s)$) {$\bullet$};
\node  at ($\l*(\cp, \s )$) {$\bullet$};
\node  at ($\l*( -0.5, \s)$) {$\bullet$};
}
\end{tikzpicture}
\] 
The corresponding gentle algebra $A(S, \ct)$ is a $3$-cycle with relations $\alpha_{2}\alpha_{1}=0$, $\alpha_{3}\alpha_{2}=0$ and $\alpha_{1}\alpha_{3}=0$. It is isomorphic to the \emph{selfinjective} algebra $I_{3}$ defined in Remark \ref{R:Main}. Hence the singularity category $\cd_{sg}(A(S, \ct))$ is triangle equivalent to the stable module category $A(S, \ct)-\ul{\mod}$, by \eqref{E:Buchweitz}. 
\end{ex}

\begin{rem}
More generally, the algebras arising as Jacobian algebras from ideal triangulations of Riemann surfaces \emph{with} punctures are often of infinite global dimension. It would be interesting to study their singularity categories and relate them to properties of the triangulation.
\end{rem}

\begin{ex} \label{E:GPs}In Example \ref{Ex:Illustrative}, the indecomposable non-projective GPs are given by:
\[
\begin{tikzpicture}[description/.style={fill=white,inner sep=2pt}]
    \matrix (n) [matrix of math nodes, row sep=0.4em,
                 column sep=1.8em, text height=1.5ex, text depth=0.25ex,
                 inner sep=0pt, nodes={inner xsep=0.3333em, inner
ysep=0.3333em}]
    {  2&3&4&7&6&5&1&2&7&8;  \\
       7&8; \\
       6&5&1&2&7&8; \\
       7&6&5&1&2&7&8; && 4; && 8. \\
       };
\draw[decoration={brace, amplitude=2mm}, decorate]($(n-4-1.west)-(13mm, -4.5mm)$) -- node[ scale=0.75, xshift=-6mm] [midway] {$c_{1}$}($(n-1-1.west) -(13mm, -2.5mm)$);

\draw[decoration={brace, amplitude=1mm}, decorate]($(n-4-1.south)-(15.5mm, 0mm)$) -- node[ scale=0.75, xshift=-6mm] [midway] {$c_{2}$}($(n-4-1.north) -(15.5mm, 0mm)$);
    
\node at ($(n-1-1)-(25pt, 0pt)$) {$R(e)=$};     
\path[->] (n-1-1) edge node[fill=white, scale=0.75, yshift=0mm] [midway] {$b$} (n-1-2);
\path[->] (n-1-2) edge node[fill=white, scale=0.75, yshift=0mm] [midway] {$c$} (n-1-3); 
\path[->] (n-1-3) edge node[fill=white, scale=0.75, yshift=0mm] [midway] {$g$} (n-1-4);
\path[->] (n-1-4) edge node[fill=white, scale=0.75, yshift=0mm] [midway] {$j$} (n-1-5); 
\path[->] (n-1-5) edge node[fill=white, scale=0.75, yshift=0mm] [midway] {$i$} (n-1-6);
\path[->] (n-1-6) edge node[fill=white, scale=0.75, yshift=0mm] [midway] {$d$} (n-1-7); 
\path[->] (n-1-7) edge node[fill=white, scale=0.75, yshift=0mm] [midway] {$a$} (n-1-8); 
\path[->] (n-1-8) edge node[fill=white, scale=0.75, yshift=0mm] [midway] {$f$} (n-1-9);
\path[->] (n-1-9) edge node[fill=white, scale=0.75, yshift=0mm] [midway] {$k$} (n-1-10); 

\node at ($(n-2-1)-(25pt, 0pt)$) {$R(f)=$};
\path[->] (n-2-1) edge node[fill=white, scale=0.75, yshift=0mm] [midway] {$k$} (n-2-2);
 
\node at ($(n-3-1)-(25pt, 0pt)$) {$R(j)=$};
\path[->] (n-3-1) edge node[fill=white, scale=0.75, yshift=0mm] [midway] {$i$} (n-3-2);
\path[->] (n-3-2) edge node[fill=white, scale=0.75, yshift=0mm] [midway] {$d$} (n-3-3); 
\path[->] (n-3-3) edge node[fill=white, scale=0.75, yshift=0mm] [midway] {$a$} (n-3-4); 
\path[->] (n-3-4) edge node[fill=white, scale=0.75, yshift=0mm] [midway] {$f$} (n-3-5);
\path[->] (n-3-5) edge node[fill=white, scale=0.75, yshift=0mm] [midway] {$k$} (n-3-6); 

\node at ($(n-4-1)-(25pt, 0pt)$) {$R(g)=$};
\path[->] (n-4-1) edge node[fill=white, scale=0.75, yshift=0mm] [midway] {$j$} (n-4-2); 
\path[->] (n-4-2) edge node[fill=white, scale=0.75, yshift=0mm] [midway] {$i$} (n-4-3);
\path[->] (n-4-3) edge node[fill=white, scale=0.75, yshift=0mm] [midway] {$d$} (n-4-4); 
\path[->] (n-4-4) edge node[fill=white, scale=0.75, yshift=0mm] [midway] {$a$} (n-4-5); 
\path[->] (n-4-5) edge node[fill=white, scale=0.75, yshift=0mm] [midway] {$f$} (n-4-6);
\path[->] (n-4-6) edge node[fill=white, scale=0.75, yshift=0mm] [midway] {$k$} (n-4-7); 

\node at ($(n-4-9)-(25pt, 0pt)$) {$R(h)=$};
\node at ($(n-4-11)-(25pt, 0pt)$) {$R(k)=$};

\end{tikzpicture}
\]
They correspond to the two cycles $c_{1}=jfe$ and $c_{2}=kgh$ in $\cc(\Lambda)$.
Theorem \ref{P:Main} yields 
\begin{align}
\cd_{sg}(\Lambda) \cong \frac{\displaystyle \cd^b(k-\mod)}{\displaystyle [3]} \oplus \frac{\displaystyle \cd^b(k-\mod)}{\displaystyle [3]}.
\end{align}
\end{ex}

\section{Proof}\label{S:Proof}
\noindent
We start with some background material on modules over gentle algebras $\Lambda=kQ/I$.  A classification of indecomposable modules over gentle algebras can be deduced from work of Ringel \cite{Ringel} (see e.g.~\cite{WW, ButlerRingel}): they are either \emph{string} or \emph{band} modules $M(w)$, where $w$ is a certain word in the alphabet $\{\alpha, \alpha^{-1} \big| \alpha \in Q_1 \}$.  Equivalently, one can consider certain quiver morphisms $\sigma\colon S \ra Q$ (for strings) and $\beta\colon B \ra Q$ (for bands), where $S$ and $B$ are of Dynkin types $\AA_{n}$ and $\widetilde{\AA}_{n}$, respectively. Then string and band modules are given as pushforwards $\sigma_{*}(M)$ and $\beta_{*}(R)$ of indecomposable $kS$-modules $M$ and indecomposable \emph{regular} $kB$-modules $R$, respectively (see e.g.~\cite{WW}).

It follows from properties (G1), (G2) \& (G4) in the Definition \ref{D:Gentle}\,Êof gentle algebras that the indecomposable projective $\Lambda$-modules are of the following form:
\begin{align}\label{E:IndProjGentle}
\begin{array}{c}
\begin{xy}
\SelectTips{cm}{}
\xymatrix@R=6pt@C=11pt{
\bullet \ar[d]_{\alpha_{1}\,} &&&&& \bullet \ar[rd]^{\, \gamma_{1}} \ar[ld]_{\beta_{1}\,}  \\
\bullet \ar@{..}[d] &&\text{or}  && \bullet  \ar@{..}[ld]&& \bullet \ar@{..}[rrdd] \\
\bullet \ar[d]_{\alpha_{l}\,} &&&\bullet \ar[ld]_{\beta_{m}\,} &&&&  \\
\bullet && \bullet &&&&&& \bullet \ar[rd]^{\, \gamma_{n}} \\
&&&& &&&& & \bullet
}
\end{xy}
\end{array}
\end{align} 
They correspond to the words $\alpha_{l}\ldots\alpha_{1}$ and $\beta_{m}\ldots\beta_{1}\gamma_{1}^{-1}\ldots\gamma_{n}^{-1}$, respectively. The definition of quiver algebras $kQ/I$ implies that the paths $\alpha_{l}\ldots\alpha_{1}$, $\beta_{m}\ldots\beta_{1}$ and $\gamma_{n}\ldots\gamma_{1}$ appearing in \eqref{E:IndProjGentle} are \emph{maximal}, e.g.~there does not exist $\alpha \in Q_1$ such that  $\alpha \alpha_l \notin I$, see for example \cite{ASS}. 

It follows from \eqref{E:IndProjGentle} that the radical $\rad P$ of an indecomposable projective $\Lambda$-module $P$ has at most two indecomposable direct summands. Moreover, \eqref{E:IndProjGentle} yields the following result about submodules of projective modules.

\begin{lem}\label{L:TechnicalV} Let $M=M(w)$ be an indecomposable $\Lambda$-module, such that $w$ contains 
\begin{align}\label{E:PictureV}
\alpha^{-1} \beta=
\begin{array}{c}
\begin{xy}
\SelectTips{cm}{}
\xymatrix@R=6pt@C=8pt{
 x \ar[rd]_(0.4){\alpha}  && z \ar[ld]^(0.4){\beta} \\
& y
}
\end{xy}
\end{array}
\end{align} 
with $\alpha \neq \beta$ as a subword. Then $M$ is not a submodule of a projective $\Lambda$-module $P$.
\end{lem}
\begin{rem}\label{R:TwoCases}
In the picture \eqref{E:PictureV}, the letters $x, y, z$ represent basis vectors of the module $M$. We do not  exclude the case $x=z$. For example, the indecomposable injective module $I_{2}$ over the Kronecker quiver $\begin{xy}\SelectTips{cm}{10}\xymatrix{1 \ar@/^/[r] \ar@/_/[r] &  2}\end{xy}$ is a \emph{string} module of the form \eqref{E:PictureV}, with pairwise different basis vectors $x, y, z$. On the other hand, the indecomposable \emph{band} modules
\[\begin{xy}\SelectTips{cm}{10}\xymatrix{k \ar@/^/[rr]^1 \ar@/_/[rr]_{\lambda} &&  k}\end{xy}\]
with $\lambda \in k^*$ correspond to the same word $\alpha^{-1} \beta$ but we have to identify $x$ and $z$ in \eqref{E:PictureV}.
 \end{rem}

Throughout the proof, we use the properties (GP1) \& (GP2) of Gorenstein projective modules over Iwanaga--Gorenstein rings, which are stated in Section \ref{S:DefMain}.

\subsection{Proof of part (a)}\label{ss:ProofParta}
Let $c \in \cc(\Lambda)$ be a cycle, which we label as follows $1 \xrightarrow{\alpha_{1}} 2 \xrightarrow{\alpha_{2}}  \cdots \xrightarrow{\alpha_{n-1}} n \xrightarrow{\alpha_{n}} 1$. Then there are short exact sequences 
\begin{align}\label{E:Shift}
0 \ra  R(\alpha_{i}) \ra P_{i} \ra R(\alpha_{i-1})\ra 0,
\end{align}
for all $i=1, \ldots, n$, where we set $\alpha_0=\alpha_n$. We give an illustration of this below.
\[
\begin{array}{c}
\begin{tikzpicture}[description/.style={fill=white,inner sep=2pt}]
    \matrix (n) [matrix of math nodes, row sep=1em,
                 column sep=2.25em, text height=1.5ex, text depth=0.25ex,
                 inner sep=0pt, nodes={inner xsep=0.3333em, inner
ysep=0.3333em}]
    {  \cdots & i &   \cdots & \, \\
        \cdots & i+1 & \cdots &\, \\
    };
\draw[->] (n-1-1) edge  node[ scale=0.75, yshift=3mm] [midway] {$\alpha_{i-1}$} (n-1-2); 
\draw[->] (n-1-2) edge  node[ scale=0.75, xshift=3mm] [midway] {$\alpha_{i}$} (n-2-2); 
\draw[->] (n-2-2) edge  node[ scale=0.75, yshift=3mm] [midway] {$\alpha_{i+1}$} (n-2-1); 

\draw[->] (n-1-2) edge  node[ scale=0.75, yshift=3mm] [midway] {$\beta_{1}$} (n-1-3); 
\draw[->] (n-1-3) edge  node[ scale=0.75, yshift=3mm] [midway] {$\beta_{n}$} (n-1-4);

\draw[->] (n-2-2) edge  node[ scale=0.75, yshift=3mm] [midway] {$\gamma_{1}$} (n-2-3); 
\draw[->] (n-2-3) edge  node[ scale=0.75, yshift=3mm] [midway] {$\gamma_{m}$} (n-2-4);

\draw[decoration={brace, amplitude=2mm}, decorate] ($(n-1-2.west) +(0, 3.5mm)$) -- node[ scale=0.75, yshift=6mm] [midway] {$R(\alpha_{i-1})$} ($(n-1-4.west)+(0, 3.5mm)$); 

\draw[decoration={brace, amplitude=2mm}, decorate] ($(n-2-4.west) +(0, -2.5mm)$) -- node[ scale=0.75, yshift=-5.5mm] [midway] {$R(\alpha_{i})$} ($(n-2-2.west)+(0, -2.5mm)$); 
 
\end{tikzpicture}
\end{array}
\]
Here $n=0$ or $m=0$ are allowed.
 In particular, \eqref{E:Shift} shows that for every $m \geq 0$ and every $i$ there is a $\Lambda$-module $X$ such that $R(\alpha_{i})$ may be written as a $m$th-syzygy module $\Omega^m(X)$. Thus,  $R(\alpha_{i}) \in \GP(\Lambda)$ by (GP2). Since projective modules are GP by definition, this shows the inclusion `$\supseteq$' in (a). 

 It remains to show that there are no further indecomposable Gorenstein projective modules. By property (GP2), we only have to consider submodules of projective modules. Using Lemma \ref{L:TechnicalV}, we can exclude all modules which correspond to a word containing $\alpha^{-1} \beta$. In particular, band modules are not Gorenstein projective - the corresponding words are cyclic and always contain subwords of the form $\alpha^{-1} \beta$. 
 
We claim that an indecomposable Gorenstein projective $\Lambda$-module $M$ containing a subword of the form $\alpha \beta^{-1}$, with $\alpha \neq \beta$ is projective. We think of $\alpha \beta^{-1}$ as a `roof'
\begin{align} \label{E:roof}
\begin{array}{c}
\begin{xy}
\SelectTips{cm}{}
\xymatrix@R=6pt@C=8pt{
& \ar[ld]_{\alpha} t \ar[rd]^{\beta} \\
s&&u
}
\end{xy}
\end{array},
\end{align}
where $s, t, u$ are basis vectors of $M$, such that $\alpha \cdot t=s$ and $\beta \cdot t =u$. 

Let $U(t) \subset M$ be the submodule generated by $t \in M$. By (GP2), $M$ is a submodule of some projective module $P$. Using this in conjunction with \eqref{E:roof} and the properties (G1) \& (G4), we see that $U(t)\cong P_{v(t)}$ is projective, where $v(t) \in Q_0$ is the vertex corresponding to $t$. If $U(t) \subsetneq M$, then $M$ 
contains a subword of the form $\alpha^{-1} \beta$, with $\alpha \neq \beta$ (we note that this statement does not use the assumption that $M$ is GP). By Lemma \ref{L:TechnicalV} this cannot happen. So we see that $M = U(t) \cong P_{v(t)}$ is indeed projective.

We have reduced the set of possible indecomposable GP $\Lambda$-modules to projective modules or directed strings $S=\beta_{n} \ldots \beta_{1}$. We also allow $S$ to consists of a single `lazy' path $e_{i}$ (this corresponds to a simple module). Let $M(S)$ be the corresponding GP $\Lambda$-module. It is contained in a projective module by (GP2). If $M(S)$ is not projective, then there exists an arrow $\alpha$ such that $\beta_{n} \ldots \beta_{1}\alpha \notin I$ and  $\gamma \beta_{n} \ldots \beta_{1}\alpha \in I$ for every arrow $\gamma\in Q_{1}$.  It follows that $M(S)=R(\alpha)$ is a direct summand of the radical of $P_{s(\alpha)}$.

{\em Claim}: If $\alpha$ does not lie on a cycle $c \in \cc(\Lambda)$, then $R(\alpha)$ has finite projective dimension. 

\noindent
If $R(\alpha)$ is not projective, then the situation locally looks as follows (we allow $n$ to be zero)
\[
\begin{array}{c}
\begin{tikzpicture}[description/.style={fill=white,inner sep=2pt}]
    \matrix (n) [matrix of math nodes, row sep=1em,
                 column sep=2.25em, text height=1.5ex, text depth=0.25ex,
                 inner sep=0pt, nodes={inner xsep=0.3333em, inner
ysep=0.3333em}]
    {  
       \cdots & \sigma & \cdots  &  \bullet \\
       ^{\displaystyle \cdots} \\
    };
\draw[->] (n-1-1) edge  node[ scale=0.75, yshift=3mm] [midway] {$\alpha$} (n-1-2); 
\draw[->] (n-1-2) edge  node[ scale=0.75, yshift=3mm] [midway] {$\beta_{1}$}  (n-1-3);
\draw[->] (n-1-3) edge  node[ scale=0.75, yshift=3mm] [midway] {$\beta_{n}$}  (n-1-4);
  
\draw[->] (n-1-2)  edge  node[scale=0.75, yshift=-1mm, xshift=3mm] [midway] {$\alpha_{1}$} (n-2-1); 

\draw[decoration={brace, amplitude=2mm}, decorate] ($(n-1-4.east) +(0, -2.5mm)$) -- node[ scale=0.75, yshift=-5.5mm] [midway] {$R(\alpha)$} ($(n-1-2.west)+(0, -2.5mm)$);

\end{tikzpicture}
\end{array}
\] 
where $\alpha_{1} \alpha \in I$. (There could be another arrow ending in $\sigma$. It is omitted from the picture since it does not affect our argument.) Moreover, $\alpha_{1}$ cannot lie on a cycle in $\cc(\Lambda)$, since this would contradict our assumption on $\alpha$. As in \eqref{E:Shift}, we have a short exact sequence
\begin{align}
0 \ra R(\alpha_{1}) \rightarrow P_{\sigma} \ra R(\alpha) \ra 0.
\end{align}
 $R(\alpha_{1})$ has the same properties as $R(\alpha)$, so we may repeat our argument. After finitely many steps, one of the occuring radical summands will be projective and the procedure stops. Indeed, otherwise we get a path $\ldots\alpha_{m} \ldots \alpha_{1}\alpha$, such that every subpath of length two is contained in $I$. Since there are only finitely many arrows in $Q$, this path is a cycle. Contradiction. Hence $R(\alpha)$ has finite projective dimension. 

Combining the claim with (GP1), we see that for arrows $\alpha$, which do not lie on a cycle in $\cc(\Lambda)$, $R(\alpha)$ is GP if and only if it is projective. Summing up, we have shown that indecomposable GP modules are either projective or direct summands $R(\alpha_{i})=\Lambda \alpha_i$ of the radical of some indecomposable projective module $P_{s(\alpha_{i})}$, where $\alpha_{i}$ is contained in a cycle $c \in \cc(\Lambda)$. This proves part (a).

\subsection{Proof of part (b)}
By Buchweitz' equivalence \eqref{E:Buchweitz}, it suffices to describe the stable category $\ul{\GP}(\Lambda)$. By part (a), the indecomposable objects in this category are precisely the radical summands $R(\alpha_{i})$ for a cycle $c=\alpha_{n}\ldots \alpha_{1} \in \cc(\Lambda)$ and \eqref{E:Shift} shows that $R(\alpha_{i})[1] \cong R(\alpha_{i-1})$. In particular, $R(\alpha_{i})[l(c)] \cong R(\alpha_{i})$. We prove \begin{align}\label{E:SemiSimple} \ul{\Hom}_{\Lambda}(R(\alpha), R(\alpha'))\cong  \delta_{\alpha \alpha'} \cdot k\end{align} below. This shows that the additive category $\ul{\GP}(\Lambda)$ is equivalent to a semisimple abelian category and therefore itself semisimple abelian. It is well-known that a semisimple abelian category with autoequivalence $[1]$ admits a unique triangulated structure with shift functor $[1]$, see e.g.~\cite[Lemma 3.4.]{Chen}. This completes the proof of part (b). The remaining part of this subsection is concerned with the proof of \eqref{E:SemiSimple}. $R(\alpha)$ is given by a  string of the following form (it starts in $\sigma$ and we allow $n=0$)
\begin{align}
\begin{array}{c}
\begin{xy}
\SelectTips{cm}{}
\xymatrix{
\ar@{..>}[r]^{\alpha} & \sigma \ar[r]^{\beta_{1}} & \ldots \ar[r]^{\beta_{n}} & \bullet.
}
\end{xy}
\end{array}
\end{align} 
Here, $\alpha$ is on a cycle $c \in \cc(\Lambda)$ and $\beta_{1}\alpha \notin I$, if $n \neq 0$. If there is a non-zero morphism of $\Lambda$-modules from $R(\alpha)$ to $R(\alpha')$, then the latter has to be a string of the following form
\begin{align}
\begin{array}{c}
\begin{xy}
\SelectTips{cm}{}
\xymatrix{
R(\alpha')\colon & \ar@{..>}[r]^{\alpha'} & \sigma' \ar[r]^{\beta'_{1}} & \ldots \ar[r]^{\beta'_{m}} & \sigma \ar[r]^{\beta_{1}} & \cdots \ar[r]^{\beta_{k}} & \bullet,
}
\end{xy}
\end{array}
\end{align} 
where we allow $k=0$ or $m=0$. If both $k$ and $m$ are zero, then (G3) and the fact that $R(\alpha')$ is a submodule of an indecomposable projective $\Lambda$-module imply that there is only one arrow starting in $\sigma$ (this arrow lies on the cycle $c$). Hence, $n=0$ and therefore $R(\alpha)=R(\alpha')$. 
$k \neq 0$ and $m=0$ imply $\alpha=\alpha'$ by (G4) -- in particular, $R(\alpha)=R(\alpha')$. 

We show that in both cases $\ul{\End}_{\Lambda}(R(\alpha)) \cong k$ holds. For this, we claim that the simple module $S_{\sigma}$ can appear (at most) twice as a composition factor of $R(\alpha)$. Indeed (G2) and the finite dimensionality of $\Lambda$ imply that every arrow of $Q$ appears at most once in the path defining $R(\alpha)$ and the arrow $\alpha$ itself does not appear at all. Now, using (G1) there is at most one arrow ending in $\sigma$ which is different from $\alpha$. This completes the proof of the claim.  
However, if $S_\sigma$ occurs twice as a composition factor, then $R(\alpha)$ locally has the following form 
\[
\sigma \xrightarrow{\beta_{1}} \cdots \xrightarrow{\neq \alpha} \sigma \xrightarrow{\gamma} \cdots \ra \bullet,
\]
where $\gamma \neq \beta_{1}$ lies on the cycle $c$, because $\alpha$ lies on this cycle with full relations, see Example \ref{E:GPs} for an illustration. In particular, this does not yield additional endomorphisms.

If $k \neq 0$ and $m \neq 0$, then it follows from (G4) that $\beta'_{m}=\alpha$. 
If $k =0$, $m \neq 0$ and $\beta'_{m} \neq\alpha$, then there are two different arrows ending in $\sigma$. Since $\alpha$ is on a cycle there is an arrow $\gamma\colon \sigma \ra \bullet$, such that $\gamma\alpha \in I$. It follows from (G3) that $\gamma \beta'_{m} \notin I$. Since $R(\alpha')=\Lambda\alpha'$ is a left ideal, the path starting in $\sigma'$ has to be maximal. In particular, it does not end in $\sigma$. Contradiction. So we again have $\beta'_{m}=\alpha$.

In both cases our morphism factors over a projective module
\begin{align}
R(\alpha) \ra P_{s(\alpha)} \ra R(\alpha')
\end{align}
and therefore $\ul{\Hom}_{\Lambda}\bigl(R(\alpha), R(\alpha')\bigr)=0$, see Example \ref{Ex:Continued} for an illustration of this case. This completes the proof.

\medskip
\noindent
\emph{Acknowledgement}. We would like to thank Jan Schr\"oer for asking us about a description of singularity categories for gentle algebras, Daniel Labardini-Fragoso for sharing his knowledge about algebras arising from surface triangulations and Igor Burban for pointing out the relation with $\mathbb{A}_{n}$-configurations of projective lines and helpful advice. We are grateful to Julian K\"ulshammer for reading a previous version of this note and for pointing out an inaccuracy in the definition of $\cc(\Lambda)$. We would like to thank the referees for carefully reading the manuscript and for making several useful suggestions, which led to an improvement of the text.

\end{document}